\documentclass{article}

\usepackage{amssymb}
\usepackage{amsthm}
\usepackage{booktabs}
\usepackage{color}
\usepackage{commath}
\usepackage[textwidth=165mm, textheight=230mm, centering]{geometry}
\usepackage{mathtools}
\usepackage{nccmath}

\newtheorem{lemma}{Lemma}[section]
\newtheorem{theorem}{Theorem}[section]

\renewcommand{\ge}{\geqslant}
\renewcommand{\le}{\leqslant}
\newcommand{\e}{\textrm e}
\renewcommand{\i}{\textrm i}

\newcommand{\tms}{\!\cdot\!}

\DeclareSymbolFont{eulargesymbols}{U}{zeuex}{m}{n}
\DeclareMathSymbol{\intop}{\mathop}{eulargesymbols}{"52}
\DeclareMathSymbol{\ointop}{\mathop}{eulargesymbols}{"49}

\begin{document}

\title{\LARGE\bf On the extreme eigenvalues and asymptotic conditioning of a class of Toeplitz matrix-sequences arising from fractional problems}

\author{M. Bogoya$^\dagger$\footnote{Dipartimento di Scienza ed Alta Tecnologia, Universit\`a dell'Insubria, Via Valleggio 11, 22100 Como (ITALY). \newline \hspace*{3.2mm}$^\times$Departamento de Matem\'aticas, CINVESTAV, M\'exico D.F. (M\'exico).\newline $^\dagger$johanmanuel.bogoya@uninsubria.it, $^\ddagger$grudsky@math.cinvestav.mx, $^\wr$mariarosa.mazza@uninsubria.it, $^\diamond$s.serracapizzano@uninsubria.it}
\and
S.M. Grudsky$^{\ddagger\times}$
\and
M. Mazza$^{\wr\ast}$
\and
S. Serra-Capizzano$^{\diamond\ast}$}
\maketitle
\date{}

\begin{abstract}
The analysis of the spectral features of a Toeplitz matrix-sequence $\left\{T_{n}(f)\right\}_{n\in\mathbb N}$, generated by a symbol $f\in L^1([-\pi,\pi])$, real-valued almost everywhere (a.e.), has been provided in great detail in the last century, as well as the study of the conditioning, when $f$ is nonnegative a.e. Here we consider a novel type of problem arising in the numerical approximation of distributed-order fractional differential equations (FDEs), where the matrices under consideration take the form
\[
\mathcal{T}_{n}=c_0T_{n}(f_0)+c_{1} h^h T_{n}(f_{1})+c_{2} h^{2h} T_{n}(f_{2})+\cdots+c_{n-1} h^{(n-1)h}T_{n}(f_{n-1}),
\]
$c_0,c_{1},\ldots, c_{n-1} \in [c_*,c^*]$, $c^*\ge c_*>0$, independent of $n$, $h=\frac{1}{n}$, $f_j\sim g_j$, $g_j=|\theta|^{2-jh}$, $j=0,\ldots,n-1$.
Since the resulting generating function depends on $n$, the standard theory cannot be applied and the analysis has to be performed using new ideas.
Few selected numerical experiments are presented, also in connection with matrices that come from distributed-order FDE problems, and the adherence with the theoretical analysis is discussed together with open questions and future investigations.
\end{abstract}


\noindent\textbf{AMS SC:}
15B05; 15A18; 26A33

\section{Introduction}

In many practical applications it is required to solve numerically linear systems of Toeplitz kind and of (very) large dimensions and hence several specialized techniques of iterative type, such as preconditioned Krylov methods and ad hoc multigrid procedures have been designed; we refer the interested reader to the books \cite{ChJi2007,Ng2004} and to the references therein. We recall that such types of large Toeplitz linear systems emerge from specific applications involving e.g. the numerical solution of (integro-) differential equations and of problems with Markov chains. On the other hand, quite recently, new examples of real world problems have emerged and among them we can cite the modeling of anomalous diffusion processes, which naturally lead to the use of fractional differential equations (FDEs). Also in this case we encounter Toeplitz structures, which are dense since the fractional operators are inherently nonlocal and hence new numerical challenges have come into the play. In all the considered applications, the sizes of resulting matrices are large and iterative solvers have to be considered, both for numerical stability and computational issues. The convergence analysis of such iterative procedures can be performed when we know the spectral features of the considered coefficient matrices. This has been done in the recent literature for certain constant-order FDE problems (see, e.g., \cite{DoMa2016, DoMa2018}), by exploiting the well-established analysis of the spectral features of Toeplitz matrix-sequences generated by Lebesgue integrable functions and the more recent Generalized Locally Toeplitz theory \cite{GaSe2017}.

Here we consider a novel type of problem arising in the numerical approximation of distributed-order fractional operators. The latter can be interpreted as a parallel distribution of derivatives of fractional orders, whose most immediate application consists in the physical modeling of systems characterized by a superposition of different processes operating in parallel. As an example, we mention the application of fractional distributed-order operators as a tool for accounting memory effects in composite materials \cite{CaFa2017} or multi-scale effects \cite{CaGi2018}. For a detailed review on the topic we refer the reader to \cite{DiPa2021}. The procedure to numerically approximate distributed-order operators is made of two steps: 1) as the distributed-order operator consists of a continuous distribution of fractional order, a numerical integration is used to discretize the distributed-order operator into a multi-term constant-order fractional derivative; 2) following the conversion of the distributed-order operator into a multi-term fractional derivative, an approximation method has to be used for discretizing each constant-order fractional derivative composing the multi-term derivative.

In this paper, we focus on the case left open in \cite{MaSe2021} which arises when the integral partition width used in step 1) is asymptotic to the discretization step adopted in 2) and where the matrices under consideration take the form
\begin{equation}\label{eq_sum}
\mathcal{T}_{n}:=c_0T_{n}(f_0)+c_{1} h^h T_{n}(f_{1})+c_{2} h^{2h}T_{n}(f_{2})+\cdots+c_{n-1} h^{(n-1)h}T_{n}(f_{n-1}),
\end{equation}
where $T_{n}(f_j)$ is the Toeplitz matrix of size $n$ generated by $f_j$, the constants $c_0,c_{1},\ldots, c_{n-1}$ belong to the interval $[c_*,c^*]$, $c^*\ge c_*>0$, and are independent of $n$, $h:=\frac{1}{n}$, $f_j\ge0$ a.e. and given $g_j(\theta):=|\theta|^{2-jh}$, $j=0,\ldots,n-1$ there exist positive constants $d_*,d^*$ independent of $j$ and $n$ satisfying
\[
 d_* g_j \le f_j \le d^* g_j \quad\text{a.e.,} \quad j=0,\ldots,n-1,
\] 
(see e.g. \cite{FaLi2018,Hu2020,MaSe2021,Li2017} and the references therein).
Taking into consideration the linearity of the operators $T_{n}(\cdot)$, $n\ge 1$, since the resulting generating function
\[
F_{n}(\theta):=\sum_{j=0}^{n-1} c_j h^{jh} f_j(\theta)
\]
depends on $n$, the standard theory cannot be applied in a straightforward manner and the analysis has to be performed using new ideas. The description and the exploitation of such new ideas is the main topic of this note.
The organization of the paper is quite simple. In Section~\ref{sec:preliminaries} we briefly recall some properties of Toeplitz matrices and the linear positive operator $T_{n}(\cdot)$ that will be used in Section~\ref{sec:main} to provide asymptotic estimates for the extreme eigenvalues of $\mathcal{T}_{n}$ in~\eqref{eq_sum}. In Section~\ref{sec:num} we show and critically discuss few selected numerical examples for illustrating our theoretical findings. Finally, Section~\ref{sec:end} is devoted to conclusions and to state few open problems.

\section{A few preliminaries}
\label{sec:preliminaries}
Given a Lebesgue integrable function $f$ defined on $[-\pi,\pi]$, i.e. $f\in L^{1}([-\pi,\pi])$, and periodically extended to the whole real line, let us consider the Toeplitz matrix $T_{n}(f)$ of size $n$ generated by $f$. For any $n$, the entries of $T_{n}(f)$ are defined via the Fourier coefficients $\{\mathfrak{a}_k(f)\}_k$, $k\in\mathbb Z$, of $f$ in the sense that
\[
\left(T_{n}(f)\right)_{s,t}=\mathfrak{a}_{s-t}(f),\quad \quad s,t \in \{1,\ldots,n\},
\]
and
\[
\mathfrak{a}_k(f):=\frac{1}{2\pi} \int_{-\pi} ^{\pi}f(\theta) \e^{-\i k \theta} \dif\theta,\quad k=0,\pm1,\pm2,\dots,\quad\i^{2}=-1.
\]
The function $f$ is called the \emph{generating function or symbol} of $T_{n}(f)$. In the general case $f$ is complex-valued and $T_{n}(f)$ is non-Hermitian for all sufficiently large $n$. However if $f$ is real-valued a.e., then $T_{n}(f)$ is Hermitian for all $n$.
Furthermore, when $f$ is real-valued and even a.e., the matrix $T_{n}(f)$ is (real) symmetric for all $n$ \cite{GaSe2017}.

As it is well known in the literature the operator $T_{n}\colon L^{1}([-\pi,\pi]) \rightarrow \mathcal M_{n}({\mathbb C})$ is a linear positive operator, briefly LPO, in the sense that:
\begin{description}
\item[Linearity.] $T_{n}(\alpha f + \beta g)=\alpha T_{n}(f) +\beta T_{n}(g)$, for any constants $\alpha,\beta \in \mathbb C$, and any functions $f,g\in L^{1}([-\pi,\pi])$.
\item[Positivity.] $T_{n}(f)$ is Hermitian nonnegative definite for any size $n\ge 1$ if $f\ge 0$ a.e.
\end{description}
For more details on LPOs, we refer the reader to \cite{Se2000} and to the beautiful book \cite{Ko1960}. Here we just recall that the linearity is a trivial fact coming from the linearity of the Fourier coefficients with respect to its argument $f$, while the second property is less obvious. In reality, as proven in \cite{Se2000}, the positivity holds even in a stronger sense, because $f\ge 0$ a.e. and $f$ not identically zero a.e. imply that $T_{n}(f)$ is positive definite for any $n\ge 1$: conversely, $f\equiv 0$ a.e. means that $T_{n}(f)$ is trivially the null matrix of size $n\ge 1$. 

Furthermore, it is a simple check to show that linearity and positivity imply monotonicity and hence from $f\ge g$ a.e. we deduce that $T_{n}(f)-T_{n}(g)$ is nonnegative definite (respectively positive definite if $f-g$ is not identically zero a.e.), that is $T_{n}(f)\ge T_{n}(g)$ ($T_{n}(f) > T_{n}(g)$ if $f-g$ is not identically zero a.e.).

We conclude the current section by fixing the notation adopted throughout the paper. 
Given a square complex matrix $B:=[b_{s,t}]_{s,t=1}^{n}$, by $\|B\|_p$ we mean the matrix norm induced by the vector $p$-norm, i.e.,
\[
\|B\|_p\coloneqq\sup\bigg\{\frac{\|Bx\|_p}{\|x\|_p}\colon x\ne0\bigg\},
\]
and precisely $\|B\|_{1}=\max_s\sum_t|b_{s,t}|$ and $\|B\|_{2}=\sigma_{\max}(B)$, also known as spectral norm. We recall that when $B$ is Hermitian or symmetric we have $\|B\|_{2}=\rho(B)$, with $\rho(B)$ being the spectral radius of $B$ and, as a consequence, that $\|B\|_{2}\le\|B\|_{1}$. Finally, when considering an invertible matrix $B$, with $\mu_{2}(B)$ we refer to the spectral condition number, i.e.,
\[
\mu_{2}(B):=\|B^{-1}\|_{2}\|B\|_{2}=\frac{\sigma_{\max}(B)}{\sigma_{\min}(B)}.
\]

\section{Analysis of the problem}
\label{sec:main}

The set of properties introduced in Section~\ref{sec:preliminaries} allows one to deduce the following general result regarding the considered class of quite involved matrices (\ref{eq_sum}), arising in the numerical approximation of distributed-order FDEs (see \cite{Hu2020}).

\begin{theorem}
For $j=0,\ldots,n-1$, let $f_j$ be an essentially bounded function defined on $[-\pi,\pi]$ such that there exist positive constants $d_*, d^*>0$, independent of $j$ and of $n$, for which
\[
d_* g_j(\theta) \le f_j(\theta) \le d^* g_j(\theta) \ \mbox{a.e.}, \ \ \ g_j(\theta) := |\theta|^{2-jh}.
\]
Set
\begin{equation}\label{eq:tildeT}
\tilde{\mathcal{T}}_{n}:=c_0T_{n}(g_0)+c_{1} h^h T_{n}(g_{1})+c_{2} h^{2h}T_{n}(g_{2})+\cdots+c_{n-1} h^{(n-1)h}T_{n}(g_{n-1}),
\end{equation}
where the constants $c_0,c_{1},\ldots, c_{n-1}$ belong to the interval $[c_*,c^*]$, $c^*\ge c_*>0$, and are independent of $n$ and $h:=\frac{1}{n}$. Then $\mathcal{T}_{n}$ as in~\eqref{eq_sum} and $\tilde{\mathcal{T}}_{n}$ are both Hermitian positive definite and in addition
\[
d_*\tilde{\mathcal{T}}_{n} \le \mathcal{T}_{n} \le d^* \tilde{\mathcal{T}}_{n}.
\]
Consequently we have
\begin{eqnarray*}
d_*\lambda_{\min}(\tilde{\mathcal{T}}_{n}) & \le \lambda_{\min}(\mathcal{T}_{n}) \le & d^* \lambda_{\min}(\tilde{\mathcal{T}}_{n}), \\
d_*\lambda_{\max}(\tilde{\mathcal{T}}_{n}) & \le \lambda_{\max}(\mathcal{T}_{n}) \le & d^* \lambda_{\max}(\tilde{\mathcal{T}}_{n}), \\
\frac{d_*}{d^*} \cdot\mu_{2}(\tilde{\mathcal{T}}_{n}) & \le \mu_{2}(\mathcal{T}_{n}) \le & \frac{d^*}{d_*} \cdot \mu_{2}(\tilde{\mathcal{T}}_{n}).
\end{eqnarray*}
\end{theorem}
\begin{proof}
Taking into account that $f_j,g_j$ are nonnegative a.e. and not identically zero a.e, from the fact $T_{n}(\cdot)$ is LPO for any size $n$ and in the strong sense, we infer that $T_{n}(f_j)$ is Hermitian positive definite for $j=0,\ldots,n-1$, and that $T_{n}(g_j)$ is real symmetric positive definite, owing to the additional fact that $g_j$ is also even for any $j=0,\ldots,n-1$. Consequently, both $\mathcal{T}_{n}$ and $\tilde{\mathcal{T}}_{n}$ are positive definite, the first generically Hermitian and the second being also real symmetric, since they are linear combinations with positive coefficients of Hermitian positive definite matrices (real symmetric positive definite matrices in the second case).

Now, any linear combination with positive coefficients of LPOs is a new LPO so that the relationships
\[
F_{n}(\theta):=\sum_{j=0}^{n-1} c_j h^{jh} f_j(\theta),\qquad \tilde F_{n}(\theta):=\sum_{j=0}^{n-1} c_j h^{jh} g_j(\theta),
\]
$\mathcal{T}_{n}=T_{n}(F_{n})$, $\tilde{\mathcal{T}}_{n}=T_{n}(\tilde F_{n})$, $d_*\tilde F_{n}(\theta) \le F_{n}(\theta) \le d^* \tilde F_{n}(\theta)$ a.e. are implied.
As a consequence, taking into account the monotonicity, we deduce 
\begin{equation}\label{fund-ineq}
d_*\tilde{\mathcal{T}}_{n} \le \mathcal{T}_{n} \le d^* \tilde{\mathcal{T}}_{n},
\end{equation}
while the other four relations are directly implied by the latter one in (\ref{fund-ineq}), with derivation steps as done in \cite{Se1998}, in a slightly different setting.
With this the proof of the theorem is concluded.
\end{proof}

The previous reduction theorem allows one to shift the focus of the considered class of problems $\mathcal{T}_{n}$, to the simpler class of matrices $\tilde{\mathcal{T}}_{n}$, with reference to the specific generating function $\tilde F_{n}(\theta)$. 

For the very same reason, since the constants $c_0,c_{1},\ldots, c_{n-1}$ belong to the interval $[c_*,c^*]$, $c^*\ge c_*>0$, and are independent of $n$, with $h=\frac{1}{n}$, the asymptotic study can be simplified further and reduced to the analysis of $\hat{\mathcal{T}}_{n}$ with $\hat{\mathcal{T}}_{n}:=T_{n}(\hat F_{n})$, and
\begin{equation}\label{eq:hatF}
\hat F_{n}(\theta):=\sum_{j=0}^{n-1} h^{jh} g_j(\theta),
\end{equation}
in accordance to the following result, whose proof mimics the same steps of the previous one.

\begin{theorem}\label{th:reduction-bis}
For $j=0,\ldots,n-1$, let $c_j$ be positive numbers for which there exist positive constants $c^*, c_*$ independent of $n$ with
\[
0<c_* \le c_j \le c^*.
\]
Set
\[
\hat{\mathcal{T}}_{n} := T_{n}(g_0)+ h^h T_{n}(g_{1})+ h^{2h}T_{n}(g_{2})+\cdots+ h^{(n-1)h}T_{n}(g_{n-1}),
\]
Then $\hat{\mathcal{T}}_{n}$ and $\tilde{\mathcal{T}}_{n}$ as in~\eqref{eq:tildeT} are both Hermitian positive definite and in addition
\[
c_*\hat{\mathcal{T}}_{n} \le \tilde{\mathcal{T}}_{n} \le c^* \hat{\mathcal{T}}_{n}.
\]
Consequently we have
\begin{eqnarray*}
c_*\lambda_{\min}(\hat{\mathcal{T}}_{n}) & \le \lambda_{\min}(\tilde{\mathcal{T}}_{n}) \le & c^* \lambda_{\min}(\hat{\mathcal{T}}_{n}), \\
c_*\lambda_{\max}(\hat{\mathcal{T}}_{n}) & \le \lambda_{\max}(\tilde{\mathcal{T}}_{n}) \le & c^* \lambda_{\max}(\hat{\mathcal{T}}_{n}), \\
\frac{c_*}{c^*} \cdot\mu_{2}(\hat{\mathcal{T}}_{n}) & \le \mu_{2}(\tilde{\mathcal{T}}_{n}) \le & \frac{c^*}{c_*} \cdot \mu_{2}(\hat{\mathcal{T}}_{n}).
\end{eqnarray*}
\end{theorem}

The rest of the section deals with the more specific case of the asymptotic spectral analysis of $\hat{\mathcal{T}}_{n}$. The next result shows asymptotic estimates for the extreme eigenvalues of $\hat{\mathcal{T}}_{n}$.
We aim at studying the asymptotic behavior of the extreme eigenvalues of $T_{n}(\hat F_{n})$ and in particular at proving that $\lambda_{\min}(T_{n}(\hat F_{n}))\sim h$. 
For reaching this goal we consider two preparatory lemmas.

\begin{lemma}\label{lm:Rnj}
For all $\alpha_j\in(0,1]$ and all natural $n$, we have
\[
T_{n}(|\theta|^{2-\alpha_j})h^{\alpha_j}\le T_{n}(|\theta|^{2})+R_{n,j},
\]
with $R_{n,j}$ a real matrix such that
\[
\rho(R_{n,j})=\|R_{n,j}\|_{2}\le\frac{\alpha_j}{3\pi n^{2}(3-\alpha_j)}.
\]
\end{lemma}

\begin{proof}
We first observe that $|\theta|^{2-\alpha_j}h^{\alpha_j}\le|\theta|^{2}$ if and only if $h\le|\theta|$, therefore
\[
|\theta|^{2-\alpha_j}h^{\alpha_j}\le|\theta|^{2}+\psi_{n,j}(\theta)
\]
where $\psi_{n,j}(\theta)\coloneqq\chi_{[-h,h]}(\theta)\big\{|\theta|^{2-\alpha_j}h^{\alpha_j}-|\theta|^{2}\big\}$. Since $T_{n}\colon L^1(-\pi,\pi)\to\mathcal{M}_{n}(\mathbb C)$ is a LPO we deduces that
\[
T_{n}(|\theta|^{2-\alpha_j})h^{\alpha_j}\le T_{n}(|\theta|^{2})+R_{n,j},
\]
where $R_{n,j}\coloneqq T_{n}(\psi_{n,j})$ which is a real and symmetric matrix since $\psi_{n,j}$ is a real-valued and even function.

Now we evaluate $\rho(R_{n,j})=\|R_{n,j}\|_{2}$ and for doing that we calculate the $k$-th Fourier coefficient of $\psi_{n,j}$, which we called $\mathfrak{a}_k(\psi_{n,j})$,
\begin{eqnarray*}
|\mathfrak{a}_k(\psi_{n,j})|=&&\bigg|\frac{1}{2\pi}\int_{-h}^{h}(|\theta|^{2-\alpha_j}h^{\alpha_j}-|\theta|^{2})\e^{-\i k\theta}\dif\theta\bigg|\\
\le&&\frac{1}{2\pi}\int_{-h}^{h}|\theta|^{2-\alpha_j}(h^{\alpha_j}-|\theta|^{\alpha_j})\dif\theta\\
=&&\frac{1}{\pi}\int_0^{h}(h^{\alpha_j}\theta^{2-\alpha_j}-\theta^{2})\dif\theta\\
=&&\frac{\alpha_j}{3\pi n^{3}(3-\alpha_j)}.
\end{eqnarray*}
Therefore,
\[
\rho(R_{n,j})=\|R_{n,j}\|_{2}\le\|R_{n,j}\|_{1}=\max_{s}\sum_{k=0}^{n-1}|\mathfrak a_{s-k}(\psi_{n,j})|\le\frac{\alpha_j}{3\pi n^{2}(3-\alpha_j)},
\]
finishing the proof.
\end{proof}

\begin{lemma}\label{lm:Mnq}
For an entire number $q\ge1$ let $M_{n,q}\coloneqq\frac{1}{q}\sum_{j=0}^{q-1}T_{n}(|\theta|^{2-\alpha_j})h^{\alpha_j}$ with $\alpha_0,\ldots,\alpha_{q-1}\in[0,1]$. Then
\[
M_{n,q}\le T_{n}(|\theta|^{2})+R_{n},
\]
where $R_{n}\coloneqq\sum_{j=0}^{q-1} R_{n,j}$ with $R_{n,j}$ as in Lemma~\ref{lm:Rnj}, and
\begin{eqnarray*}
\lambda_{\min}(M_{n,q})\le&&\lambda_{\min}(T_{n}(|\theta|^{2}))+\frac{h^{2}}{\pi q}\sum_{j=0}^{q-1}\frac{\alpha_j}{\alpha_j+1}\\
\sim&&h^{2}.
\end{eqnarray*}
\end{lemma}

\begin{proof}
Let $x_0\in\mathbb{R}^n$ be an eigenvector of $T_{n}(|\theta|^{2})$ related to $\lambda_0\coloneqq\lambda_{\min}(T_{n}(|\theta|^{2}))$, that is $T_{n}(|\theta|^{2})x_0=\lambda_0 x_0$. Assume also that $\|x_0\|_{2}=1$. From Lemma~\ref{lm:Rnj} we know that
\begin{eqnarray*}
M_{n,q}=&&\frac{1}{q}\sum_{j=0}^{q-1} T_{n}(|\theta|^{2-\alpha_j})h^{\alpha_j}\\
\le&&T_{n}(|\theta|^{2})+\frac{1}{q}\sum_{j=0}^{q-1} R_{n,j}\\
=&&T_{n}(|\theta|^{2})+\frac{1}{q}R_{n}.
\end{eqnarray*}
Hence we obtain
\begin{eqnarray*}
\lambda_{\min}(M_{n,q})=&&\min_{\|x\|_{2}=1} x^* M_{n,q} x \le x_0^* M_{n,q} x_0\\
\le&&x_0^* T_{n}(|\theta|^{2}) x_0+\frac{1}{q} x_0^* R_{n} x_0\\
=&&\lambda_0+\frac{1}{q}\rho(R_{n})\\
=&&\lambda_0+\frac{1}{q}\sum_{j=0}^{q-1}\|R_{n,j}\|_{2}\\
\le&&\lambda_0+\frac{\alpha_j}{3\pi (3-\alpha_j)}h^{2}.
\end{eqnarray*}
Finally, using that $\lambda_0\sim h^{2}$ we finish the proof.
\end{proof}

\begin{theorem}
Let $M_{n}\coloneqq hT_{n}(\hat F_{n})$. Then
\[
\lambda_{\min}(M_{n})\le\lambda_{\min}(T_{n}(|\theta|^{2}))+\frac{h^{3}}{\pi}\sum_{j=0}^{n-1}\frac{jh}{jh+1}\sim h^{2}.
\]
\end{theorem}

\begin{proof}
It is sufficient to invoke the Lemma~\ref{lm:Mnq} with $q=n$ and $\alpha_j= jh$.
\end{proof}

\section{Numerical experiments}
\label{sec:num}

In this section we numerically check the theoretical findings provided in Section~\ref{sec:main} and provide an heuristic analysis. With this aim we first calculate the entries of the matrices $T_{n}(\hat F_{n})$ and $T_{n}(\eta)$ for $\eta(\theta):=\theta^{2}$. From~\eqref{eq:hatF}, the $k$-th Fourier coefficient of the function $\hat F_{n}(\theta)$ named $\mathfrak{a}_k(\hat F_{n})$ can be obtained as
\[
\mathfrak{a}_k(\hat F_{n})=\sum_{j=0}^{n-1}h^{jh}\mathfrak{a}_k(g_j),
\]
where $\mathfrak{a}_k(g_j)$ is the $k$-th Fourier coefficient of the function $g_j(\theta)=|\theta|^{2-jh}$. Additionally, $\mathfrak{a}_k(g_j)$ can be exactly calculated with
\begin{align*}
\mathfrak{a}_k(g_j)=&\frac{1}{\pi}\int_0^\pi \theta^{2-jh}\cos(k\theta)\dif\theta\\
=&\frac{\pi^{2-jh}}{3-jh}\ {}_{1}H_{2}\Big(\frac{1}{2}(3-jh);\frac{1}{2},\frac{1}{2}(5-jh);-\frac{k^{2}\pi^{2}}{4}\Big)
\end{align*}
where ${}_p H_q$ is the well-known Generalized Hypergeometric function which can be found as an internal function, e.g., in the \textit{Mathematica} software. Thus we can exactly calculate the entries of the Toeplitz matrix $T_{n}(\hat F_{n})$.

On the other hand, the Fourier coefficients of $\eta$ can be exactly calculated by
\[
\mathfrak{a}_k(\eta)=\begin{cases}\mfrac{2}{k^{2}}\cos(\pi k)+\mfrac{k^{2}\pi^{2}-2}{k^{3}\pi}\sin(\pi k), & k\ne0;\\[2mm]
\mfrac{\pi^{2}}{3}, & k=0.\end{cases}
\]
We now present an heuristic finding which is related to the Avram--Parter formula. Let $q\in\mathbb N$ be a sufficiently large number, and consider the related symbol $\hat F_{q}$, that is
\[\hat F_{q}(\theta)\coloneqq\sum_{j=0}^{q-1} \frac{|\theta|^{2-\frac{j}{q}}}{q^{\frac{1}{q}}}=\theta^{2}\frac{1-\frac{1}{q|\theta|}}{1-\big(\frac{1}{q|\theta|}\big)^{\frac{1}{q}}}.\]
According to \cite[Th.3.2]{BoBoGrMa2015b}, for every $x\in(0,1)$, the Avram--Parter formula shows that
\begin{equation}\label{eq:AbLim}
\lim_{n\to\infty}\lambda_{\lceil xn\rceil}(T_{n}(\hat F_{q}))=Q_{q}(x),
\end{equation}
where $Q_{q}\colon[0,\pi]\to\mathbb R$ is the quantile function related to $\hat F_{q}$ which in this case is given by $Q_{q}(x)\coloneqq \hat F_{q}(\pi x)$. Using the estimation
\[\Big\{1-\Big(\frac{1}{\pi jqh}\Big)^{\frac{1}{q}}\Big\}^{-1}=\frac{q}{\log(\pi jqh)}\Big\{1+O\Big(\frac{1}{q}\log(\pi jqh)\Big)\Big\}\quad\mbox{as}\quad q\to\infty,\]
a simple asymptotic calculation produces
\begin{eqnarray*}
Q_{q}(jh)&=&(\pi j h)^{2}\frac{1-\frac{1}{\pi j q h}}{1-\big(\frac{1}{\pi j q h}\big)^{\frac{1}{q}}}\\
&=&\frac{\pi jh(\pi jqh-1)}{\log(\pi jqh)}+O\Big((\pi jh)^{2}+\frac{\pi jh}{q}\Big).
\end{eqnarray*}
Abusing of the limit~\eqref{eq:AbLim}, let's take $q=n$ and $x=jh$ to obtain
\[\lambda_{j}(T_{n}(\hat F_{n}))\approx Q_{n}(jh)=\frac{\pi j(\pi j-1)}{\log(\pi j)}h+O\big((jh)^{2}\big),\]
which gives us the heuristic calculations
\begin{equation}\label{eq:Heu}
\lambda_{\min}(T_{n}(\hat F_{n}))\approx\frac{\pi(\pi-1)}{\log(\pi)}h+O(h^{2})\quad\mbox{and}\quad
\lambda_{\max}(T_{n}(\hat F_{n}))\approx 
\frac{\pi^{2}n}{\log(\pi n)}+O(1).
\end{equation}
Table~\ref{tb:SphatF} shows the extreme eigenvalues of $T_{n}(\hat F_{n})$ for different values of $n$. It agrees with the heuristic results in~\eqref{eq:Heu} and, in particular, shows that $\mu_{2}(T_{n}(\hat F_{n}))=O\big(\frac{(\pi n)^{2}}{\log(\pi n)}\big)$. Interestingly, the constant $\frac{\pi(\pi-1)}{\log(\pi)}\approx5.8774$ agrees with the row $\lambda^*_{\min}$ since we expect a very slow convergence in this cases.

\renewcommand{\arraystretch}{1.1}
\begin{table}[ht]
\centering
{\footnotesize\begin{tabular}{ccccccc}
\toprule
$n$ & 64 & 128 & 256 & 512 & 1024 &2048\\ \midrule
$\lambda_{\min}$ & $7.8737\tms10^{-2}$ & $3.9480\tms10^{-2}$ & $1.9768\tms10^{-2}$ & $9.8914\tms10^{-3}$ & $4.9476\tms10^{-3}$ & $2.4743\tms10^{-3}$\\
$\lambda^*_{\min}$ & $5.0392$ & $5.0534$ & $5.0607$ & $5.0644$ & $5.0663$ & $5.0673$\\
$\lambda_{\max}$ & $120.9373$ & $212.8457$ & $380.1275$ & $687.1094$ & $1\,254.2460$ & $2\,307.9670$\\
$\lambda^*_{\max}$ & $1.015$ & $1.010$ & $1.006$ & $1.004$ & $1.002$ & $1.001$\\
$\mu_{2}$ & $1\,535.9667$ & $5\,391.2724$ & $19\,229.1665$ & $69\,465.1987$ & $253\,507.4186$ & $932\,790.7960$ \\
$\mu^*_{2}$ & $0.2015$ & $0.1999$ & $0.1989$ & $0.1982$ & $0.1978$ & $0.1976$ \\ \bottomrule
\end{tabular}}
\caption{The extreme eigenvalues and the conditioning of $T_{n}(\hat F_{n})$ for different values of $n$. Here we used the notation $\lambda^*_{\min}=n\lambda_{\min}$, $\lambda^*_{\max}=\frac{\log(\pi n)}{\pi^{2} n}\lambda_{\max}$, and $\mu^*_{2}=\frac{\log(\pi n)}{(\pi n)^{2}}\mu_{2}$.}\label{tb:SphatF}
\end{table}

Having in mind that matrices in the form of $T_{n}(\hat F_{n})$ arise within distributed-order fractional problems which are typically ill-conditioned, we conclude this section providing a short discussion on the conditioning of $T_{n}^{-1}(\eta)T_{n}(\hat F_{n})$, i.e., we check how the spectrum changes when preconditioning with a Laplacian-like matrix. Table~\ref{tb:SpCond} shows the extreme eigenvalues of $T_{n}^{-1}(\eta)T_{n}(\hat F_{n})$ for different values of $n$
and it suggests that
\[
\lambda_{\min}(T_{n}^{-1}(\eta)T_{n}(\hat F_{n}))=O\Big(\frac{n}{\log n}\Big)\quad\mbox{and}\quad\lambda_{\max}(T_{n}^{-1}(\eta)T_{n}(\hat F_{n}))=O\Big(\frac{n^{2}}{\log(\pi n)}\Big).
\]
From this results we expect that $\mu_{2}(T_{n}^{-1}(\eta)T_{n}(\hat F_{n}))=O(n)$ which means that $\eta(\theta)$ is not acting as an effective preconditioning function. Figure~\ref{fg:SphatF} shows the eigenvalues of $T_{n}(\hat F_{n})$ and of $T_{n}^{-1}(\eta)T_{n}(\hat F_{n})$ for $n=1024$ and from them, it is immediately clear that the largest eigenvalue of $T_{n}(\hat F_{n})$ is magnified after preconditioning and that the preconditioned spectrum is far from being well-clustered.

\renewcommand{\arraystretch}{1.1}
\begin{table}[ht]
\centering
{\footnotesize\begin{tabular}{ccccccc}
\toprule
$n$ & 64 & 128 & 256 & 512 & 1024 & 2048\\ \midrule
$\lambda_{\min}$ & 15.4546 & 26.5447 & 46.4058 & 82.3378 & 147.9148 & 268.6040\\
$\lambda^{\dagger}_{\min}$ & 1.004 & 1.006 & 1.005 & 1.003 & 1.001 & 1.000\\
$\lambda_{\max}$ & 1\,793.0355 & 6\,479.2722 & 23\,557.6771 & 86\,165.4914 & 316\,954.9557 & 1\,175\,952.6713\\
$\lambda^{\dagger}_{\max}$ & 2.3217 & 2.3715 & 2.4048 & 2.4268 & 2.4412 & 2.4587\\
$\mu_{2}$ & 116.0198 & 244.0896 & 507.6456 & 1\,046.4872 & 2\,142.8207 & 4\,378.0162\\
$\mu^{\dagger}_{2}$ & 1.8128 &  1.9069 & 1.9830 & 2.0439 & 2.0926 & 2.1377\\ \bottomrule
\end{tabular}}
\caption{The extreme eigenvalues and the conditioning of $T_{n}^{-1}(\eta)T_{n}(\hat F_{n})$ where $\eta(\theta)=\theta^{2}$, for different values of $n$. Here we used the notation $\lambda^{\dagger}_{\min}=\frac{\log n}{n}\lambda_{\min}$, $\lambda^{\dagger}_{\max}=\frac{\log(\pi n)}{n^{2}}\lambda_{\max}$, and $\mu^{\dagger}_{2}=\frac{1}{n}\mu_{2}$.}\label{tb:SpCond}
\end{table}

\begin{figure}[ht]
\centering
\includegraphics[width=150mm]{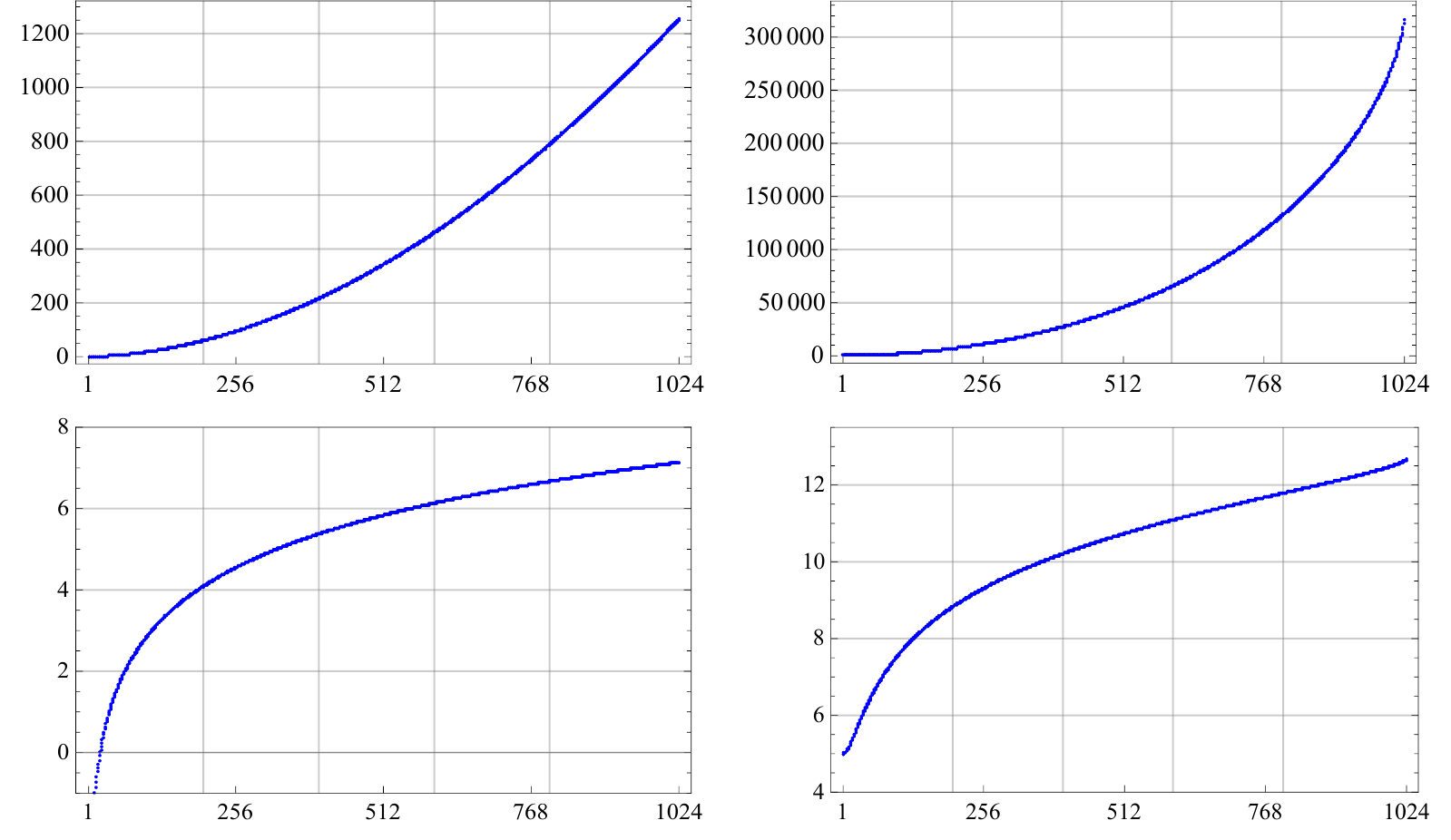}
\caption{The eigenvalues of $T_{n}(\hat F_{n})$ and of $T_{n}^{-1}(\eta)T_{n}(\hat F_{n})$ where $\eta(\theta)= \theta^{2}$, for $n=1024$. Regular scale (top) and logarithmic scale (bottom).}\label{fg:SphatF}
\end{figure}

\newpage
\section{Conclusions}
\label{sec:end}
In this note we have considered a novel type of spectral problem involving Toeplitz structures and arising in the numerical approximation of distributed-order fractional differential equations. The spectral study of the matrices under consideration has been reduced to the study of the matrix
\[
T_{n}(g_0)+ h^h T_{n}(g_{1})+ h^{2h}T_{n}(g_{2})+\cdots+ h^{(n-1)h}T_{n}(g_{n-1}),
\]
where $h=\frac{1}{n}$, $g_j(\theta)=|\theta|^{2-jh}$, and $j=0,\ldots,n-1$.
Because the resulting generating function depends on $n$, the standard theory cannot be applied and the analysis has been performed using new ideas. Few selected numerical experiments have been presented and critically discussed. Many open questions remain: we can consider for instance the use of the given spectral information in a context of Krylov preconditioning or multigrid design  for large linear systems coming from FDEs with distributed order and the extension of the proposed analysis in a multidimensional context, when the spatial derivatives belong to a $d$-dimensional domain with $d\ge 2$. These lines of research will be the subject of future investigations.

\section*{Acknowledgement}
The third and fourth authors are members of the INdAM research group GNCS. The work of the third author was partly supported by the GNCS-INdAM Young Researcher Project 2020 titled \lq\lq Numerical methods for image restoration and cultural heritage deterioration''.

\bibliographystyle{acm}
\bibliography{Toeplitz}

\end{document}